\documentclass{amsart}
\usepackage{textcomp}
\usepackage{amsmath}
\usepackage{mathrsfs}
\usepackage{amsthm}
\usepackage[linktocpage]{hyperref}
\usepackage{amsfonts,graphics,amsthm,amsfonts,amscd,latexsym}
\usepackage{epsfig}
\usepackage{flafter}
\usepackage{mathtools}
\usepackage{comment}
\usepackage{stmaryrd}
\usepackage{enumitem}
\usepackage{epigraph}

\hypersetup{
colorlinks=true,
linkcolor=red,
citecolor=green,
filecolor=blue,
urlcolor=blue
}
\usepackage{tikz}
\usetikzlibrary{graphs,positioning,arrows,shapes.misc,decorations.pathmorphing}

\tikzset{
>=stealth,
every picture/.style={thick},
graphs/every graph/.style={empty nodes},
}

\tikzstyle{vertex}=[
draw,
circle,
fill=black,
inner sep=1pt,
minimum width=5pt,
]
\usepackage[position=top]{subfig}
\usepackage[alphabetic]{amsrefs}
\usepackage{amssymb}
\usepackage{color}

\calclayout

\usetikzlibrary{decorations.pathmorphing}
\tikzstyle{printersafe}=[decoration={snake,amplitude=0pt}]

\newcommand{\ord}{\operatorname{ord}}

\newcommand{\vol}{\operatorname{vol}}

\def\O#1.{\mathcal {O}_{#1}}
\def\pr #1.{\mathbb P^{#1}}
\def\af #1.{\mathbb A^{#1}}
\def\ses#1.#2.#3.{0\to #1\to #2\to #3 \to 0}
\def\xrar#1.{\xrightarrow{#1}}
\def\K#1.{K_{#1}}
\def\bA#1.{\mathbf{A}_{#1}}
\def\bM#1.{\mathbf{M}_{#1}}
\def\bL#1.{\mathbf{L}_{#1}}
\def\bB#1.{\mathbf{B}_{#1}}
\def\bK#1.{\mathbf{K}_{#1}}
\def\subs#1.{_{#1}}
\def\sups#1.{^{#1}}

\usepackage{tikz}
\usetikzlibrary{matrix,arrows,decorations.pathmorphing}

\newtheorem{theorem}{Theorem}[section]
\newtheorem{lemma}[theorem]{Lemma}
\newtheorem{proposition}[theorem]{Proposition}
\newtheorem{corollary}[theorem]{Corollary}

\theoremstyle{definition}
\newtheorem{definition}[theorem]{Definition}
\newtheorem{example}[theorem]{Example}

\theoremstyle{remark}

\numberwithin{equation}{section}

\usepackage[all]{xy}

\newcounter{rownumber}[figure]
\newcounter{rownumber-irr}[figure]
\newcounter{rownumber-p1}[figure]
\begin{document}

\title{K-stability of Fano threefolds of rank $4$ and degree $24$}
\begin{abstract}
We prove that all smooth Fano threefolds of rank $4$ and degree $24$ are K-stable.
\end{abstract}

\author[G.~Belousov]{Grigory Belousov}
\address{Bauman Moscow State Technical University, Moscow, Russia}
\email{belousov\_grigory@mail.ru}

\author[K.~Loginov]{Konstantin Loginov}
\address{Steklov Mathematical Institute of Russian Academy of Sciences, Moscow, Russia}
\email{loginov@mi-ras.ru}

\maketitle

\section{Introduction}

We work over the field of complex numbers. Three-dimensional smooth Fano varieties are known to belong to $105$ families. In \cite{ChAll}, the problem of characterising K-stable Fano threefolds was solved for a general element in each of these families. In particular, it was proven that a general smooth Fano threefolds with Picard rank $4$ and of degree $24$ is K-stable, see \cite[Lemma 4.3.2]{ChAll}. In this paper, we prove the following

\begin{theorem}
\label{main-theorem}
All smooth Fano threefolds with Picard rank $4$ and of degree $24$ are K-stable.
\end{theorem}

%Sketch of the proof.
%To prove this theorem, we apply Abban-Zhuang theory \cite{AZ20} in the form \cite[Section 1.7]{ChAll}.

\textbf{Acknowledgements.}
The authors thank Ivan Cheltsov for proposing the problem and for useful discussions. The second author is supported by Russian Science Foundation under grant 21-71-00112.

\section{Preliminary results}

Throughout the paper, we shall use the following notation. Let $X$ be a smooth divisor of type $(1,1,1,1)$ in $(\mathbb{P}^1)^4=\mathbb{P}^1\times\mathbb{P}^1\times\mathbb{P}^1\times\mathbb{P}^1$. Then $X$ is a Fano threefold with $\rho(X)=4$ and $(-K_X)^3=24$, and conversely, all such Fano threefolds are realised in this way. Let
\[
\pi_i\colon X\rightarrow\mathbb{P}^1, \quad \quad \quad \quad \quad \pi_{i, j}\colon X\to (\mathbb{P}^1)^2, \quad \quad \quad \quad \quad \pi_{i, j, k}\colon X\to (\mathbb{P}^1)^3
\]
be the projections to the product of factors of $(\mathbb{P}^1)^4$ with numbers $i$, $j$ and $k$, for $1\leq i<j<k\leq 4$. Note that the map $\mathrm{Pic}(\mathbb{P}^1)^4\to\mathrm{Pic}(X)$ is an isomorphism by Lefschetz theorem. If a divisor $D$ on $(\mathbb{P}^1)^4$ has type $(a_1, a_2, a_3, a_4)$ for $a_i\in\mathbb{Z}$, we will say that $D|_X$ also has type $(a_1, a_2, a_3, a_4)$. We start with the description of elementary geometric properties of $X$.

\begin{proposition}
\label{prop-cones}
\begin{enumerate}[label=(\alph*)]
\item
Each contraction $\pi_{i, j, k}$, for $1\leq i<j<k\leq 4$, is a blow up of a smooth elliptic curve $C=C_{i,j,k}$ on $(\mathbb{P}^1)^3$ which is an intersection of two divisors of type $(1,1,1)$. For an exceptional divisor $S=S_{i,j,k}$ of $\pi_{i, j, k}$, we have $S\simeq C\times \mathbb{P}^1$.

% is a smooth ruled surfaces over $C_i$, and has the type $()$

\item
Each contraction $\pi_{i, j}$ is a conic bundle (although, it is not extremal) whose discriminant curve is a smooth elliptic curve  of type $(2, 2)$ in $(\mathbb{P}^1)^2$.%\times\mathbb{P}^1$.

\item
Each contraction $\pi_i$ is a sextic del Pezzo fibration (although, it is not extremal).

\item
The nef cone of $X$ is generated by surfaces $Y_i$, for $i=1,2,3,4$, where $Y_i$ are fibers of the contractions $\pi_i\colon X\to \mathbb{P}^1$. Moreover, one has $Y_i \cdot l_{j,k,l} = 1$ if $\{i, j, k, l\}=\{1, 2, 3, 4\}$, and $0$ otherwise, where $l_{i,j,k}$ are fibers of birational contractions $\pi_{i, j, k}$ for $1\leq i<j<k\leq 4$.

\item
The Mori cone $\overline{\mathrm{NE}}(X)$ is generated by the curves $l_{i,j,k}$ for $1\leq i<j<k\leq 4$.%, where $l_{i,j,k}$ are fibers of birational contractions $\pi_{i, j, k}$.%\colon X\to (\mathbb{P}^1)^3$.

\item
The pseudo-effective cone (which coincides with the movable cone) of $X$ is generated by the surfaces $S_{i,j,k}$ which are exceptional for $\pi_{i, j, k}$, for $1\leq i<j<k\leq 4$, and the fibers $Y_i$ of the projections $\pi_i$. %Each $S_i$ is a ruled surface over an elliptic curve.
\end{enumerate}
\end{proposition}
\begin{proof}
\begin{enumerate}[label=(\alph*), wide, labelwidth=!]
\item
Let $[(x_1: y_1), (x_2: y_2), (x_3: y_3), (x_4: y_4)]$ be the coordinates on $(\mathbb{P}^1)^4$. Without loss of generality, assume that $i, j, k = 2, 3, 4$. Let $X$ be given by an equation $x_1L + y_1L' = 0$ in $(\mathbb{P}^1)^4$ where $L=L(x_2,y_2,x_3,y_3,x_4,y_4)$ and $L'=L'(x_2,y_2,x_3,y_3,x_4,y_4)$ are forms of type $(1,1,1)$. The map $\pi_{2,3,4}$ is an isomorphism outside the curve $C=C_{2,3,4}$ given by the equations $L=L'=0$ on $(\mathbb{P}^1)^3$. Note that $C$ is smooth  since $X$ is smooth. The fibers of $\pi_{2,3,4}$ over points in $C$ are projective lines. Hence, $X$ is a blow up of $C$ on $(\mathbb{P}^1)^3$.  Let us denote the exceptional divisor of $\pi_{2,3,4}$ by $S=S_{2,3,4}$. Note that $S\simeq C\times \mathbb{P}^1$, since $S$ is given by the equations $L=L'=0$ on $(\mathbb{P}^1)^4$.

Also note that $S$ has type $(-1, 1, 1, 1)$. Indeed, on $X$ the equation $L=0$, which has type $(0, 1, 1, 1)$, defines a union of two prime divisors: the divisor $S$ given by the equations $L=L'=0$, and the divisor $Y_1$ given by the equation $y_0=0$ which has degree $(1,0,0,0)$.

\item
Follows from the previous assertion and the fact that the projection $(\mathbb{P}^1)^3\to (\mathbb{P}^1)^2$ is a $\mathbb{P}^1$-bundle. We may assume that $i,j=1,2$. In the above notation, consider a curve $C=C_{1,2,3}$ which is the image of the exceptional divisor $S_{1,2,3}$ under the map $\pi_{1,2,3}$. Then $C=D_1\cap D_2$ on $(\mathbb{P}^1)^3$ where both $D_i$ have type $(1,1,1)$. Note that $D_i$ are smooth del Pezzo surfaces of degree $6$, and $C=D_2|_{D_1}$ is an anticanonical section on $D_1$.

Consider an induced contraction $\pi_{1,2}|_{D_i}\colon D_i\to (\mathbb{P}^1)^2$ which is a blow down of two disjoint $(-1)$-curves. Since $C$ is equivalent to $-K_{D_1}$, it intersects each $(-1)$-curve in one point. Hence, $\pi_{1,2}|_{C}$ is an isomorphism, and $\pi_{1,2}(C)=C'$ is a smooth elliptic curve on $(\mathbb{P}^1)^2$, so it has bidegree $(2, 2)$. Note that the fiber of $\pi_{1,2}\colon X\to (\mathbb{P}^1)^2$ over points of $C'$ is a union of two smooth rational curves intersecting at a point.

\item
By adjunction formula, the general fiber $Y=Y_i$ of $\pi_i\colon X\to \mathbb{P}^1$ is a del Pezzo surface. It is smooth, since $X$ is smooth. The fact that $K_Y^2=6$ is straightforward.

\item
The claim on the intersection numbers $Y_i\cdot l_{j, k, l}$ is straighforward. By Lefschetz theorem, an embedding $X\to (\mathbb{P}^1)^4$ induces an isomorphism of Picard groups, thus $\mathrm{N}_1(X)\simeq\mathbb{R}^4$. Note that the curves $l_{i, j, k}$, for $1\leq i<j<k\leq 4$, which are fibers of extremal contractions $\pi_{i.j,k}$, for $1\leq i<j<k\leq 4$, are extremal. Hence they define support functions for the nef cone. As a consequence, the nef cone is contained in the cone generated by the divisors $Y_i$ of types $(1,0,0,0)$, $(0,1,0,0)$, $(0,0,0,1)$, $(0,0,0,1)$, which are fibers of the contractions $\pi_i$ for $1\leq i\leq 4$. On the other hand, clearly such divisors are nef.

\item
Since the Mori cone is dual to the nef cone, the claim follows immediately. %Note that
%\[
%l_1\cdot (1, 0, 0, 0) = 1, \quad \quad l_1\cdot (0, 1, 0, 0) = l_1\cdot (0, 0, 1, 0) = l_1\cdot (0, 0, 0, 1) = 0.
%\]

\item
Obviously, the cone spanned by $S_{i,j,k}$ for $1\leq i<j<k\leq 4$, and $Y_i$ for $1\leq i\leq 4$, is contained in the pseudo-effective cone. On the other hand, if $D$ is an effective prime divisor different from $S_{i,j,k}$ for $1\leq i<j<k\leq 4$, we have $D\cdot l_{i,j,k}\geq 0$ for $1\leq i<j<k\leq 4$. Consequently, $D$ has type $(a_1, a_2, a_3, a_4)$ with $a_i\geq0$. The claim follows.
\end{enumerate}
\end{proof}

The following fact is well-known.

\begin{proposition}[{\cite[Lemma 8.15]{ChAll}}]
The group $\mathrm{Aut}(X)$ is finite.
\end{proposition}

Throughout the paper, by a $(-1)$-curve we mean a smooth rational curve such that its anti-canonical degree is equal to $1$.
\begin{proposition}
\label{prop-du-val}
%Let $X$ be a smooth divisor of type $(1,1,1,1)$ in $\mathbb{P}^1\times\mathbb{P}^1\times\mathbb{P}^1\times\mathbb{P}^1$.
Let $Y=Y_i\subset X$ be a fiber of the projection $\pi_i\colon X\to \mathbb{P}^1$ for some $1\leq i\leq 4$. Then~$Y$ is either a smooth sextic del Pezzo surface, or a normal sextic del Pezzo surface with exactly one du Val singularity of type $\mathbb{A}_1$. Moreover, in the latter case we have $\rho(Y)=3$, and $Y$ contains exactly $3$ $(-1)$-curves.
\end{proposition}
\begin{proof}
Without loss of generality, we assume that $i=1$. First we show that $Y$ is irreducible and reduced. In the notation of Proposition \ref{prop-cones}, the projection $\pi_{2,3,4}$ is a blow up of a smooth elliptic curve $C=C_{2,3,4}$ given by the equations $L=L'=0$ of type $(1,1,1)$ on $(\mathbb{P}^1)^3$, if $X$ is given by the equation $x_1L + y_1L' = 0$ in $(\mathbb{P}^1)^4$. Then, any fiber of $\pi_{1}$ is given by the equation $\lambda L + \mu L'=0$  in $(\mathbb{P}^1)^3$ where $\lambda, \mu\in\mathbb{C}$. After a coordinate change, we may assume that $\mu= 0$. Hence $Y$ is given by the equation $L=0$.

Assume that $Y=\sum k_j Y_j$ for $j\geq 1$ and some $k_j\geq 1$. Since the divisor $(L'=0)$ is ample on $(\mathbb{P}^1)^3$, the intersection $(L'=0)\cap Y_i$ is a non-empty curve for each~$j$. Since $(L'=0)\cap Y=C$, and $C$ is smooth, we conclude that $Y$ has only one component, so $j=1$ and $k_1=1$. In other words, $Y$ is irreducible and reduced. Moreover, the same argument shows that the singular locus of $Y$ cannot be one-dimensional, hence $Y$ is normal.
%Consider the projection $Y\subset \mathbb{P}^1\times\mathbb{P}^1\times\mathbb{P}^1\to\mathbb{P}^1\times\mathbb{P}^1$ to the last two copies of $\mathbb{P}^1$. Suppose that $Y$ is given by the equation
%\[x_0 l(x_1, y_1, x_2, y_2) + y_0 l'(x_1, y_1, x_2, y_2)=0,\]
%where $l$ and $l'$ have bidegree $(1, 1)$. Then, if $l\neq 0$ and $l'\neq 0$, then $(x_0:y_0)$ is determined uniquely. If $l=0$ and $l'=0$ share no common components, then $Y$ is birational to $\mathbb{P}^1\times \mathbb{P}^1$. Otherwise, we may assume that they share a component of type $(1,1)$, and hence after scaling $l=l'$, or that they share a component of type $(1,0)$, and we have $l=x_1l_1$, $l'=x_1l'_1$.
It follows that $Y$ is a normal Gorenstein surface with ample $-K_Y$ and $(-K_Y)^2=6$. By \cite{HW81} %and \cite{R94, F95},
we have the following possibilities:
\begin{itemize}
\item
$Y$ is rational, in which case $Y$ has du Val singularities,

\item
$Y$ is non-rational, in which case $Y$ is a cone over a curve of genus $1$. However, this case does not occur on a smooth threefold $X$, see e.g. \cite[Proposition 1.3]{Lo19}.%degree $6$ embedded in $\mathbb{P}^5$. Hence
%\item
%$Y$ is non-normal and rational, in which case its normalization $\overline{Y}$ is isomorphic either to a Hirzebruch surface $\mathbb{F}_2$, or to a Hirzebruch surface $\mathbb{F}_4$. This case is excluded above.
\end{itemize}

It remains to deal with the first case. That is, we assume that $Y$ is a sextic del Pezzo surface with du Val singularities. Denote by $Y'\to Y$ the minimal resolution of $Y$, and consider the blow down morphism $Y'\to\mathbb{P}^2$. According to \cite{HW81}, we have the following possibilities:
\begin{enumerate}
\item
$Y$ is smooth, in which case $Y'=Y$, $Y'$ is a blow up of $3$ non-collinear points on $\mathbb{P}^2$, $\rho(Y)=4$, and $Y'$ has $6$ $(-1)$-curves.

\item
$Y$ has a unique singular point of type $\mathbb{A}_1$, $Y'$ is a blow up of a point and a two infinitely near points on $\mathbb{P}^2$ such that these the $3$ points are not collinear. In this case $\rho(Y)=3$, and $Y'$ has $4$ $(-1)$-curves.

\item
$Y$ has a unique singular point of type $\mathbb{A}_1$, $Y'$ is a blow up of $3$ collinear points on $\mathbb{P}^2$. In this case $\rho(Y)=3$, and $Y'$ has $3$ $(-1)$-curves.

\item
$Y$ has one singular point of types $\mathbb{A}_2$, $Y'$ is a blow up of $3$ infinitely near points, $\rho(Y)=2$, and $Y'$ has $2$ $(-1)$-curves.

\item
$Y$ has two singular points of types $\mathbb{A}_1$, $Y'$ is a blow up of a point and two infinitely near points, $\rho(Y)=2$, and $Y'$ has $2$ $(-1)$-curves.

\item
$Y$ has two singular points of type $\mathbb{A}_1$ and $\mathbb{A}_2$, $Y'$ is a blow up of three infinitely near points, $\rho(Y)=1$, and $Y'$ has $1$ $(-1)$-curve.
\end{enumerate}

Note that since $Y\subset (\mathbb{P}^1)^3$, we have $\rho(Y)\geq 3$. This excludes the cases $(4)-(6)$. We show that the possibility $(2)$ also does not occur. Consider a contraction $Y\to (\mathbb{P}^1)^2$. Since $\rho(Y)=3$, only one curve can be contracted. Note that this curve should pass through a singular point on~$Y$. This contradicts to the fact that on $Y$ there exists a $(-1)$-curve which is disjoint from the singular point on $Y$, while on $\mathbb{P}^1\times \mathbb{P}^1$ there are no $(-1)$-curves. The proposition is proven.
\end{proof}

\begin{proposition}
\label{St2}
Let $P$ be a point on $X$. Then at most one of the fibers of $\pi_i\colon X\to \mathbb{P}^1$ that contain $P$ is singular at $P$.
\end{proposition}
\begin{proof}
Local computation. Let $[(x_1: y_1), (x_2: y_2), (x_3: y_3), (x_4: y_4)]$ be the coordinates on $(\mathbb{P}^1)^4$. We may assume that $P=[(0: 1), (0: 1), (0: 1), (0: 1)]$. Let $Y_i$ be the fiber of $\pi_i$ that contains the point $P$. Consider an affine chart $\mathbb{A}^4\subset(\mathbb{P}^1)^4$ given by the equations
$y_i\neq 0%,\quad y_2\neq 0,\quad y_3\neq 0,\quad y_4\neq 0.
$ for $1\leq i\leq 4$. Let $
z_i=x_i/y_i%,\quad y=\frac{x_2}{y_2},\quad z=\frac{x_3}{y_3},\quad t=\frac{x_4}{y_4}
$
be the coordinates on $\mathbb{A}^4$. Then $\left.X\right|_{\mathbb{A}^4}$ is given by the equation
\[
g(x, y, z, t) = c_1z_1 + c_2z_2 + c_3z_3 + c_4z_4 + g_{\geq 2}(z_1, z_2, z_3, z_4) = 0, \quad \quad c_i\in\mathbb{C},
\]
where $g_{\geq 2}$ has degree $\geq 2$. Note that around $P$ we have
$
Y_i|_{\mathbb{A}^4}=\{ z_i=0\}%\quad Y_2|_{\mathbb{A}^4}=\{ y=0\}, \quad Y_3|_{\mathbb{A}^4}=\{ z=0\}, \quad Y_4|_{\mathbb{A}^4}=\{ t=0\}.
$ for $1\leq i\leq 4$.
Since $X$ is smooth, we see that at most one of $Y_i$ is singular at $P$.
\end{proof}

\begin{proposition}
\label{Sing1}
%Let $X\subset\mathbb{P}^1\times\mathbb{P}^1\times\mathbb{P}^1\times\mathbb{P}^1$ be a smooth Fano threefold such that $X$ is a divisor of type $(1,1,1,1)$.
Let $P$ be a point on $X$ and let $Y_i$ be the fibers of $\pi_i\colon X\rightarrow\mathbb{P}^1$ that contain $P$, for $1\leq i\leq 4$. Assume that for some $i$ there exists a $(-1)$-curve $E_1$ on $Y_i$  such that $P\in E_1$. Then for some $j$ with $1\leq j\leq 4$, the fiber $Y_j$ is smooth, and moreover, the inclusion $E_1\subset Y_j$ holds.
\end{proposition}
\begin{proof}
We may assume that there exists a $(-1)$-curve $E_1$ on $Y_1$ that passes through $P$. Assume that $Y_1$ has a singular point $P'$ (maybe $P=P'$). Then $P'$ is the intersection point of three $(-1)$-curves $E_1,E_2,E_3$ on $Y_1$. Consider three projections $\pi_{2,3}, \pi_{2,4}, \pi_{3,4}$ that map $Y_1$ onto $(\mathbb{P}^1)^2$. We see that two of these projections contract $E_1$. Hence, we may assume that $Y_2$ contains $E_1$ and $E_2$. Since~$Y_2$ is smooth at $P'$ (see Proposition \ref{St2}), we conclude that $Y_2$ is smooth.
\end{proof}
\iffalse
\begin{corollary}
\label{cor-smooth-fiber}
Let $Y=Y_i$ be a singular fiber of the projection $\pi_i\colon X\to \mathbb{P}^1$, and let $E_1$ be a $(-1)$-curve on $Y$. Then there exists a smooth fiber $Y'=Y_j$ of some other projection $\pi_j\colon X\to \mathbb{P}^1$ such that $E_1\subset Y'$.
\end{corollary}
\fi

\begin{proposition}
\label{Sing2}
%Let $X\subset\mathbb{P}^1\times\mathbb{P}^1\times\mathbb{P}^1\times\mathbb{P}^1$ be a smooth Fano threefold such that $X$ is a divisor of type $(1,1,1,1)$.
Let $P$ be a point on $X$ and let $Y_i$ be the fibers of $\pi_i\colon X\rightarrow\mathbb{P}^1$ that contain $P$, for $1\leq i\leq 4$. Put $Y=Y_1$. Assume that the curve $C'=Y\cap S$ passes through $P$, where $S=S_{2,3,4}$ is the exceptional divisor of the contraction $\pi_{2,3,4}$. Then at least one of $Y_i$ is smooth.
\end{proposition}
\begin{proof}
Note that $S$ is an exceptional divisor of the projection $\pi_{2,3,4}\colon X\rightarrow(\mathbb{P}^1)^3$. Since $C'$ passes through $P$, we see that there exists a fiber $f$ of ruled surface $S$ such that $f$ passes through $P$. Note that $Y_2$ contains $f$. We see that $f$ is a $(-1)$-curve on $Y_2$. By Proposition \ref{Sing1} we see that at least one of $Y_i$ is smooth.
\end{proof}

\iffalse
\begin{example}
\[
X = \{ x_0 x_1 x_2 x_3 + y_0 y_1 y_2 y_3 + \ldots = 0 \}\subset \mathbb{P}^1\times\mathbb{P}^1\times\mathbb{P}^1\times\mathbb{P}^1.
\]

It was proven in \cite{ChAll} that this variety is K-stable.
\end{example}
\fi
%Since $C_1,C_2,C_3,C_4\subset X$, we see that $$g(x,0,0,0)=g(0,y,0,0)=g(0,0,z,0)=g(0,0,0,t)=0.$$ So, $X$ has a singularity in $P$, a contradiction.
%Assume that $X_1$ is singular in $P$. Put $$C_1=\mathbb{P}^1\times(0,1)\times(0,1)\times(0,1),\quad C_2=(0,1)\times\mathbb{P}^1\times(0,1)\times(0,1),$$ $$C_3=(0,1)\times(0,1)\times\mathbb{P}^1\times(0,1),\quad C_4=(0,1)\times(0,1)\times(0,1)\times\mathbb{P}^1.$$ Since $X\cdot C_1=X\cdot C_2=X\cdot C_3=X\cdot C_4=1$, we see that $C_2,C_3,C_4\subset X_1$. Let $f_2\colon X\rightarrow\mathbb{P}^1$ be the projection on the second factor and $X_2$ be the fiber that contains $P$. Assume that $X_2$ is singular in $P$. As above, $C_1,C_3,C_4\subset X_2$. So, $C_1,C_2,C_3,C_4\subset X$.

\section{K-stability and Abban-Zhuang theory}
We briefly recall some of the definitions in the theory of K-stability. For more details, see a survey \cite{Xu21} and references therein.

\subsection{Discrepancies and thresholds}
Let $(X, D)$ be a pair, and let $f\colon Y\to X$ be a proper birational morphism from a normal variety $Y$. For a prime divisor $E$ on $Y$, a \emph{log discrepancy} of $E$ with respect to the pair $(X, D)$ is defined as
\[
A_{(X, D)}(E) = 1 + \mathrm{coeff}_E ( K_Y - f^*(K_X + D)).
\]

\iffalse
\emph{Log canonical threshold} of an effective $\mathbb{Q}$-divisor $B$ with respect to a klt pair $(X, D)$ is defined as
\[
\mathrm{lct}(X, D; B) = \sup \{ \lambda\ |\ K_X + D + \lambda B\ \text{is log canonical} \}.
\]

Then, \emph{the global log canonical threshold} for a klt log Fano pair $(X, D)$ is defined as
\[
\mathrm{lct}(X, D) = \inf \{ \mathrm{lct}(X, D; B)\ |\ B\sim_{\mathbb{Q}}-K_X-D\ \text{and}\ B\ \text{is an effective $\mathbb{Q}$-divisor} \}.
\]
\fi
Put $L=-K_X-D$. By a \emph{pseudo-effective threshold} of $E$ with respect to a log Fano pair $(X, D)$ we mean the number
\[
\tau_{(X, D)}(E) = \sup\{ x\in\mathbb{R}_{\geq0}\colon f^*L - xE\ \text{is pseudo-effective}\}.
\]

Similarly, we define the \emph{nef threshold} of $E$ with respect to a log Fano pair $(X, D)$:
\[
\epsilon_{(X, D)}(E) = \sup\{ x\in\mathbb{R}_{\geq0}\colon f^*L - xE\ \text{is nef}\}.
\]

The \emph{expected vanishing order} of $E$ with respect to a log Fano pair $(X, D)$ is
\[
S_{(X, D)}(E) = \frac{1}{\mathrm{vol}(L)} \int_{0}^{\infty}{\mathrm{vol}( f^*L - xE )dx},
\]
where $\mathrm{vol}$ is the volume function, see \cite{Laz04}. The \emph{beta-invariant} $\beta_{(X, D)}(E)$ of $E$ with respect to a log Fano pair $(X, D)$ is defined as follows:
\[
\beta_{(X, D)}(E) = A_{(X, D)}(E) - S_{(X, D)}(E).
\]

Recall that the $\delta$-invariant of $E$ with respect to a log Fano pair $(X, D)$ (resp., $\delta$-invariant of $E$ along $Z$ with respect to a log Fano pair $(X, D)$) are defined as
\[
\delta(X, D) = \inf_{E/X} \frac{A_{(X, D)}(E)}{S_{(X, D)}(E)}, \ \ \ \ \delta_Z(X, D) = \inf_{E/X,\ Z\subset C(E)} \frac{A_{(X, D)}(E)}{S_{(X, D)}(E)}
\]
where $E$ runs through all prime divisors over $X$ (resp., $E$ runs through all prime divisors over $X$ whose center contains $Z$).
\iffalse
\subsection{Delta-invariant via log canonical thresholds.}
By \cite[Theorem C]{BJ20}, the following definition of an delta-invariant, given by Fujita-Odaka in ..., is equivalent to the one given above. Let $(X, D)$ be a klt log Fano pair. For every $m\in \mathbb{Z}_{\geq0}$ such that $m(K_X+D)$ is Cartier divisor, consider the space $\mathrm{H}^0(X, -m(K_X+D))$. Let $N_m = h^0(X, -m(K_X+D))$. Then, for every basis $s_1, \ldots, s_{N_m}$ in $\mathrm{H}^0(X, -m(K_X+D))$, we denote by $D(s_1), \ldots, D(s_{N_m})$ the corresponding divisors.
We call an effective $\mathbb{Q}$-divisor
\[
B = \frac{1}{mN_m}\left(D(s_1) + \ldots + D(s_{N_m})\right)
\]
an \emph{$m$-basis type divisor}. Note that $D\sim_{\mathbb{Q}} -K_X$. For each $m$, we define
\[
\delta_m(X, D) = \inf \{\mathrm{lct}(X, D; B)\ | \ B\ \text{is of $m$-basis type} \}.
\]

For each $m$, this infimum is attained, see e.g. ... Then, the delta-invariant for the pair $(X, D)$ is defined as follows:
\[
\delta(X, D) = \limsup_{m} \delta_m(X, D).
\]

This limit exists by \cite[Theorem E]{BJ20}.

\fi

%We formulate the main definitions.

\begin{definition}[\cite{Li17}, \cite{Fu19}, \cite{Fu16}]
\label{def-k-stability}
A klt log Fano pair $(X, D)$ is called
\begin{enumerate}
%\item
%\emph{K-semistable} if $\beta_{(X, D)}(E)\geq0$ for any prime divisor $E$ over $X$,
\item
\emph{divisorially semistable} (resp., \emph{divisorially stable}), if $\beta_{(X, D)}(E)\geq0$ (resp., $\beta_{(X, D)}(E)>0$) for any prime divisor $E$ on $X$. We say that $X$ is \emph{divisorially unstable} if it is not divisorially semistable,
\item
\emph{K-semistable} (resp., \emph{K-stable}) if $\beta_{(X, D)}(E)\geq0$ (resp., $\beta_{(X, D)}(E)>0$) for any prime divisor $E$ over $X$. We say that $X$ is \emph{K-unstable} if it is not K-semistable,
\item
\emph{uniformly K-stable} if $\delta(X, D)>1$.
%\item
%\emph{divisorially semistable}, if $\beta_{(X, D)}(E)\geq0$ for any prime divisor $E$ on $X$,
\end{enumerate}
\end{definition}

This definition is equivalent the original definition of K-stability given in \cite{Tia97}, \cite{Don02} in terms of test configurations. If $D=0$, we avoid writing it for the clarity of notation. Now, we recall two propositions from Abban-Zhuang theory developed in \cite{AZ20} and \cite[1.7]{ChAll}.

\begin{proposition}[{\cite[Corollary 1.7.26]{ChAll}}]
\label{Van1}
Let $X$ be a smooth Fano threefold, $Y\subset X$ be an irreducible normal surface that has at most du Val singularities, $Z\subset Y$ be an irreducible smooth curve. Then for any prime divisor $E$ over $X$ such that $C(E)=Z$ we have
\begin{equation}
\label{ineq-curve}
\frac{A(E)}{S(E)}\geq \min\left\{ \frac{1}{S_X(Y)}, \frac{1}{S(W^{Y}_{\bullet,\bullet};Z)} \right\},
\end{equation}
where
\begin{multline*}
S(W^{Y}_{\bullet,\bullet};Z)=\frac{3}{(-K_X)^3}\int\limits_0^{\infty}(P(u)^2\cdot Y)\cdot\ord_Z(N(u)|_Y)du
+\frac{3}{(-K_X)^3}\int\limits_0^{\infty}\int\limits_0^{\infty}\vol(P(u)|_Y-vZ)dv du.
\end{multline*}
Moreover, if the equality holds in \eqref{ineq-curve}, then $\frac{A(E)}{S(E)} = \frac{1}{S_X(Y)}$.
\end{proposition}

%By \cite[Remark 1.7.28]{ChAll}, in the above proposition the surface $Y$ may be non-normal, if $Z$ is an irreducible curve in $Y$ such that $Z\not\subset \mathrm{Sing}(Y)$. In this case, we should replace both $P(u)|_Y$ and $N(u)|_Y$ by their pullbacks on the normalization of $Y$.

Let $P(u,v)$ be the positive part of the Zariski decomposition of $P(u)|_Y-vZ$, and $N(u,v)$ be the negative
part of the Zariski decomposition of this divisor. We can write $N(u)|_Y=dZ+N'_Y(u)$, where $Z\not\subset
\mathrm{Supp}(N'_Y(u))$ and $d=d(u)=\ord_Z(N(u)|_Y)$.

\begin{proposition}[{\cite[Theorem 1.7.30]{ChAll}}]
\label{Van2}
Let $X$ be a smooth Fano threefold, $Y\subset X$ be an irreducible normal surface that has at most du Val singularities, $Z\subset Y$ be an irreducible smooth curve such that the log pair $(Y,Z)$ has purely log
terminal singularities. Let $P$ be a point in the curve $Z$. Then
\begin{equation}
\delta_P(X)\geq\min\left\{\frac{1-\Delta_P}{S(W^{Y,Z}_{\bullet,\bullet,\bullet};P)},\frac{1}{S(V^Y_{\bullet,\bullet};Z)},\frac{1}{S_X(Y)}\right\},
\end{equation}
where $\Delta_P$ is the different of the log pair $(Y,Z)$, and
\begin{multline*}
S(W^{Y,Z}_{\bullet,\bullet,\bullet};P)=\frac{3}{(-K_X)^3}\int\limits_0^{\infty}\int\limits_0^{\infty}(P(u,v)\cdot Z)^2dv du+\\
+\frac{6}{(-K_X)^3}\int\limits_0^{\infty}\int\limits_0^{\infty}(P(u,v)\cdot Z)\cdot\ord_P(N'_Y(u)|_Z+N(u,v)|_Z)dv du.
\end{multline*}
Moreover, if the inequality is an equality and there exists a prime divisor $E$ over
the threefold such that $C_X(E)=P$ and $\delta_P(X)=\frac{A(E)}{S(E)}$ then $\delta_P(X)=\frac{1}{S_X(Y)}$.
\end{proposition} 

\section{Divisorial stability}
\label{sec-divisorial-stability}
For the reader's convenience, we show that $X$ is divisorially stable, where $X$ is a smooth divisor of type $(1,1,1,1)$ in $(\mathbb{P}^1)^4$. Consider a divisor $-K_X-uY$, where $Y$ is a fiber of $\pi_1\colon X\rightarrow\mathbb{P}^1$, so $Y$ is of type $(1,0,0,0)$. We put $S=S_{2,3,4}$ to be exceptional divisor of the contraction  $\pi_{2,3,4}$.

\begin{proposition}
\label{ZD}
Let $-K_X-uY=P(u)+N(u)$ be the Zariski decomposition. Then
\[
P(u)=
\begin{cases}-K_X-uY,\quad\quad\quad\quad\quad\quad\text{for}\ \ 0\leq u\leq1, \\
-K_X-uY-(u-1)S,\quad\text{for}\ \ 1\leq u\leq2.
\end{cases}
\] where $S$ is the exceptional divisor of the birational morphism $\pi_{2,3,4}\colon X\rightarrow(\mathbb{P}^1)^3$.%\times\mathbb{P}^1\times\mathbb{P}^1$.
\end{proposition}
\begin{proof}
Note that $-K_X-uY$ is ample for $0\leq u<1$. Also note that for $1\leq u\leq 2$, we have
\[
-K_X-uY \equiv -K_X-uY-(u-1)S + (u-1)S,
\]
where $P(u) = -K_X-uY-(u-1)S\equiv (0, 2-u, 2-u, 2-u)$ is nef and numerically trivial on the rulings of $S$.
\end{proof}

%For the reader's convenience, we give a proof of the fact that a Fano threefold 4.1 is divisorially stable.

\begin{proposition}[{\cite{Fu16}}]
\label{St1}
%Let $X\subset\mathbb{P}^1\times\mathbb{P}^1\times\mathbb{P}^1\times\mathbb{P}^1$ be a smooth Fano threefold such that $X$ is a divisor of type $(1,1,1,1)$. Then
The variety $X$ is divisorially stable.
\end{proposition}
\begin{proof}
%we should check that $\eta(D)> 0$ for $D$ is of types $(k_1,k_2,k_3,k_4)$, where $ k_i=0,1$.
Using notation as in Proposition \ref{prop-cones}, we know that the pseudo-effective cone is generated by divisors $S_{i,j,k}$ and $Y_i$.
%\[
%(-1,1,1,1), (1,-1,1,1), (1,1,-1,1), (1,1,1,-1), (1,0,0,0), (0,1,0,0), (0,0,1,0), (0,0,0,1),
%\]
By {\cite[Lemma 9.5, Remark 9.6]{Fu16}} it is enough to consider only the divisors $D$ such that $-K_X-D$ is big. Then, up to permutation of factors on $\mathbb{P}^1$, it is enough to consider only one divisor $D$ of type $(1,0,0,0)$. We compute $\beta_X(D)=1-\frac{S_X(D)}{(-K_X)^3}$. We have
\[
\frac{1}{(-K_X)^3}\int\limits_0^{1}\vol(-K_X-tD)dt=\frac{1}{24}\int\limits_0^1(-K_X-tD)^3dt=\frac{1}{24}\int\limits_0^1(24-18t)dt=\frac{15}{24}.
\]
%Hence, $\eta(D)> 0$.
%We have $Y\sim (1, 0, 0, 0)$, and $C\sim(-1, 1, 1, 1)$.
For $1\leq u\leq 2$, we have $-K_X-uY-(u-1)S\equiv (0, 2-u, 2-u, 2-u)$. Compute
\begin{multline*}
\frac{1}{(-K_X)^3}\int_1^{2}{\mathrm{vol}(-K_X-uY-(u-1)S)du}
= \frac{1}{24}\int_1^{2}{\mathrm{vol}(0, 2-u, 2-u, 2-u)du} \\
= \frac{1}{24}\int_1^{2}{(0, 2-u, 2-u, 2-u)^3(1,1,1,1)du}
%%= \int_1^{2}{(2-u)^3(0,1,1,1)^3(1,1,1,1)du}
= \frac{1}{24}\int_1^{2}{6(2-u)^3du} = \frac{3}{48}.
\end{multline*}

Thus, we obtain $S_X(D)=\frac{33}{48}$, and hence $\beta_X(D) = 1-\frac{33}{48}>0$, hence $X$ is divisorially stable.
\iffalse
Assume that $D$ is of type $(1,1,0,0)$. Then $$\int\limits_0^{+\infty}\vol(-K_X-tD)dt=\int\limits_0^1(-K_X-tD)^3dt=\int\limits_0^1(24-36t+12t^2)dt=10<24.$$ Hence, $\eta(D)> 0$.

Assume that $D$ is of type $(1,1,1,0)$. Then $$\int\limits_0^{+\infty}\vol(-K_X-tD)dt=\int\limits_0^1(-K_X-tD)^3dt=\int\limits_0^1(24-54t+36t^2-6t^3)dt=\frac{15}{2}<24.$$ Hence, $\eta(D)> 0$. Assume that $D$ is of type $(1,1,1,1)$. Then $$\int\limits_0^{+\infty}\vol(-K_X-tD)dt=\int\limits_0^1(-K_X-tD)^3dt=\int\limits_0^1 24(1-t)^3dt=6<24.$$ Hence, $\eta(D)> 0$.
\fi
\end{proof}

%Hence, in our case K-polystability implies K-stability.

\section{Computations}
\label{sec-center-point}
In this section, we work in the following setting. Assume that $X$ is a smooth threefold of type $(1,1,1,1)$ in $(\mathbb{P}^1)^4$, and that $P$ is a point in $X$. Let $E$ be a divisor over $X$ such that its center $C(E)$ contains the point $P\in X$. We consider a fiber $Y\ni P$ of a projection $\pi_i\colon X\to \mathbb{P}^1$. Without loss of generality, we may assume that $i=1$, so $Y\simeq (1, 0, 0, 0)$. We shall estimate beta-invariants $\beta_X(E)$ using Propositions \ref{Van1} and \ref{Van2}. Consider several possibilities.

\subsection{$Y$ is a smooth surface, and there are no $(-1)$-curves on $Y$ that pass through $P$. }
\label{sec-center-point-1}
Let $Z$ be a fiber of the induced contraction
\[
\pi_{1, 2}|_Y\colon Y\to\mathbb{P}^1\times\mathbb{P}^1
\]
that passes through $P$. Then $Z$ is a $0$-curve in $Y$. Assume that $0\leq u\leq 1$. By Proposition \ref{ZD} we have $P(u)=-K_X-uY$. So, for $0\leq v\leq 1$ we have
\[
P(u,v)=(-K_X-Y)|_Y-vZ=-K_Y-vZ.
\]

Then
\[
\int\limits_0^{1}\int\limits_0^{1}(P(u,v)\cdot Z)^2dv du=\int\limits_0^{1}\int\limits_0^{1}4dv du=4.
\]

Assume that $1\leq v\leq 2$. Since there exist exactly two $(-1)$-curves $E_1,E_2$ that meet $Z$, we see that
\[
P(u,v)=-K_Y-vZ-(v-1)(E_1+E_2), \quad \quad N(u, v) = (v-1)(E_1+E_2).
\]

Then
\[
\int\limits_0^{1}\int\limits_1^{2}(P(u,v)\cdot Z)^2dv du=\int\limits_0^{1}\int\limits_1^{2}(4-2v)^2dvdu=\frac{4}{3}.
\]

Note that $P(u,v)=0$ for $v>2$.

\

Assume that $1\leq u\leq 2$. By Proposition \ref{ZD} we have $P(u)=-K_X-uY-(u-1)S$. Put $C'=S|_Y$. Note that $C'$ meets every $(-1)$-curve in $Y$ in one point and $C'^2=6$, $C'\cdot Z=2$. So, $P(u,v)=-K_Y-vZ-(u-1)C'$ for $0\leq v\leq 2-u$. Then
\[
\int\limits_1^{2}\int\limits_0^{2-u}(P(u,v)\cdot Z)^2dv du=\int\limits_1^{2}\int\limits_0^{2-u}(4-2u)^2dvdu=1.
\]

Assume that $2-u\leq v\leq 4-2u$. Then
\[
P(u,v)=-K_Y-vZ-(u-1)C'-(u+v-2)(E_1+E_2), \quad \quad N(u, v) = (u+v-2)(E_1+E_2).
\]

Then
\[
\int\limits_1^{2}\int\limits_{2-u}^{4-2u}(P(u,v)\cdot Z)^2dv du=\int\limits_1^{2}\int\limits_{2-u}^{4-2u}(8-4u-2v)^2dvdu=\frac{1}{3}.
\]

Note that $P(u,v)=0$ for $v>4-2u$. Hence,
\[
\frac{3}{(-K_X)^3}\int\limits_0^{\infty}\int\limits_0^{\infty}(P(u,v)\cdot Z)^2dv du=\frac{1}{8}\left(4+\frac{4}{3}+1+\frac{1}{3}\right)=\frac{5}{6}.
\]

Since, by our assumption, there are no $(-1)$-curves that pass through $P$, so $\ord_P N(u,v)|_Z=~0$. Moreover, $N'_Y(u)|_Z=0$ for $0\leq u\leq 1$ and $\ord_P N(u,v)|_Z\leq u-1$ for $1\leq u\leq 2$. Compute %\color{red}{It seems that $\ord_P N(u,v)|_Z = 0$ since $N(u,v)|_Z = (u+v-2)(E_1+E_2)$, and $E_i$ do not pass through $P$.}\color{black}
\begin{multline*}
\int\limits_0^{\infty}\int\limits_0^{\infty}(P(u,v)\cdot Z)\cdot\ord_P(N'_Y(u)|_Z+N(u,v)|_Z)dv du \\
\leq \int\limits_1^2\int\limits_0^{2-u}(4-2u)(u-1)dvdu+\int\limits_1^2\int\limits_{2-u}^{4-2u}(8-4u-2v)(u-1)dvdu \\
=\frac{1}{6}+\frac{1}{12}=\frac{1}{4}.
\end{multline*}

So,
\[
\frac{6}{(-K_X)^3}\int\limits_0^{\infty}\int\limits_0^{\infty}(P(u,v)\cdot Z)\cdot\ord_P(N'_Y(u)|_Z+N(u,v)|_Z)dv du\leq \frac{1}{16}.
\]

Hence,
\begin{equation}
S(W^{Y,Z}_{\bullet,\bullet,\bullet};P)\leq\frac{5}{6}+\frac{1}{16}=\frac{43}{48}<1.
\end{equation}

Note that $\ord_Z(N(u)|_Y)=0$. So,
\[
S(W^{Y}_{\bullet,\bullet};Z)=\frac{3}{(-K_X)^3}\int\limits_0^{\infty}\int\limits_0^{\infty}\vol(P(u)|_Y-vZ)dv du.
\]

Assume that $0\leq u\leq 1$. Then
\begin{multline*}
\int\limits_0^{\infty}\int\limits_0^{\infty}\vol(P(u)|_Y-vZ)dv du=\int\limits_0^{2}\vol(-K_Y-vZ)dv\\
=\int\limits_0^{1}(-K_Y-vZ)^2dv+\int\limits_1^{2}(-K_Y-vZ-(v-1)(E_1+E_2))^2dv\\
=\int\limits_0^{1}(6-4v)dv+\int\limits_1^{2}(6-4v+2(v-1)^2)dv=4+\frac{2}{3}=\frac{14}{3}.
\end{multline*}

Assume that $1\leq u\leq 2$. Then
\begin{multline*}
\int\limits_0^{\infty}\int\limits_0^{\infty}\vol(P(u)|_Y-vZ)dv du=\int\limits_1^{2}\int\limits_0^{4-2u}\vol(-K_Y-vZ-(u-1)C')dvdu\\
=\int\limits_1^{2}\int\limits_0^{2-u}(-K_Y-vZ-(u-1)C')^2dvdu\\
+\int\limits_1^{2}\int\limits_{2-u}^{4-2u}(-K_Y-vZ-(u-1)C'-(u+v-2)(E_1+E_2))^2dvdu\\
=\int\limits_1^{2}\int\limits_0^{2-u}(6+6(u-1)^2-12(u-1)-4v+4(u-1)v)dvdu\\
+\int\limits_1^{2}\int\limits_{2-u}^{4-2u}(6+6(u-1)^2-12(u-1)-4v+4(u-1)v+2(u+v-2)^2)dvdu \\
=1+\frac{1}{6} = \frac{7}{6}.
\end{multline*}

Then
\begin{equation}
\label{ineq-z-irred-fiber}
S(W^{Y}_{\bullet,\bullet};Z)=\frac{41}{48}<1.
\end{equation}

Hence, $\delta_P(X)>1$.

\subsection{$Y$ is a smooth surface, and there exists a unique $(-1)$-curve $E_1$ on $Y$ that passes through $P$}
\label{sec-center-point-2}
Put $Z=E_1$. Assume that $0\leq u\leq 1$. By Proposition \ref{ZD} we have $P(u)=-K_X-uY$. So, for $0\leq v\leq 1$
\[
P(u,v)=(-K_X-Y)|_Y-vZ=-K_Y-vZ.
\]

Then
\[
\int\limits_0^{1}\int\limits_0^{1}(P(u,v)\cdot Z)^2dv du=\int\limits_0^{1}\int\limits_0^{1}(1+v)^2dv du=\frac{7}{3}.
\]

Assume that $1\leq v\leq 2$. Since there exist exactly two $(-1)$-curves $E_2,E_3$ that meet $Z$, we see that
\[
P(u,v)=-K_Y-vZ-(v-1)(E_2+E_3), \quad \quad N(u, v) = (v-1)(E_2+E_3).
\]

Then
\[
\int\limits_0^{1}\int\limits_1^{2}(P(u,v)\cdot Z)^2dv du=\int\limits_0^{1}\int\limits_1^{2}(3-v)^2dvdu=\frac{7}{3}.
\]

Note that $P(u,v)=0$ for $v>2$. Now assume that $1\leq u\leq 2$. By Proposition \ref{ZD} we have
\[
P(u)=-K_X-uY-(u-1)S, \quad \quad N(u) = (u-1) S.
\]

Put $C'=S|_Y$. Note that $C'$ meets every $(-1)$-curve in $Y$ in one point and $C'^2=6$. So, $P(u,v)=-K_Y-vZ-(u-1)C'$ for $0\leq v\leq 2-u$. Then
\[
\int\limits_1^{2}\int\limits_0^{2-u}(P(u,v)\cdot Z)^2dv du=\int\limits_1^{2}\int\limits_0^{2-u}(2+v-u)^2dvdu=\frac{7}{12}.
\]

Assume that $2-u\leq v\leq 4-2u$. Then
\[
P(u,v)=-K_Y-vZ-(u-1)C'-(u+v-2)(E_2+E_3), \quad \quad N(u, v) = (u+v-2)(E_2+E_3).
\]

Then
\[
\int\limits_1^{2}\int\limits_{2-u}^{4-2u}(P(u,v)\cdot Z)^2dv du=\int\limits_1^{2}\int\limits_{2-u}^{4-2u}(6-3u-v)^2dvdu=\frac{7}{12}.
\]

Note that $P(u,v)=0$ for $v>4-2u$. Hence,
\[
\frac{3}{(-K_X)^3}\int\limits_0^{\infty}\int\limits_0^{\infty}(P(u,v)\cdot Z)^2dv du=\frac{1}{8} \left(\frac{7}{3}+\frac{7}{3}+\frac{7}{12}+\frac{7}{12}\right)=\frac{35}{48}.
\]

Since there exists a unique $(-1)$-curve that passes through $P$, we see that $\ord_P N(u,v)|_Z=0$. Moreover, $N'_Y(u)|_Z=0$ for $0\leq u\leq 1$, $N'_Y(u)|_Z=u-1$ for $1\leq u\leq 2$, hence $\ord_P N'_Y(u)|_Z\leq u-1$ for $1\leq u\leq 2$\color{black}. Then
\begin{multline*}
\int\limits_0^{\infty}\int\limits_0^{\infty}(P(u,v)\cdot Z)\cdot\ord_P(N'_Y(u)|_Z+N(u,v)|_Z)dv du \\
\leq\int\limits_1^2\int\limits_0^{2-u}(2+v-u)(u-1)dvdu+\int\limits_1^2\int\limits_{2-u}^{4-2u}(6-3u-v)(u-1)dvdu \\
=\frac{1}{8}+\frac{1}{8}=\frac{1}{4}.
\end{multline*}

So,
\[
\frac{6}{(-K_X)^3}\int\limits_0^{\infty}\int\limits_0^{\infty}(P(u,v)\cdot Z)\cdot\ord(N'_Y(u)|_Z+N(u,v)|_Z)dv du\leq \frac{1}{16}.
\]

Hence,
\begin{equation}
S(W^{Y,Z}_{\bullet,\bullet,\bullet};P)\leq\frac{35}{48}+\frac{1}{4}=\frac{47}{48}<1.
\end{equation}

Note that $\ord_Z(N(u)|_Y)=0$. So,
\[
S(W^{Y}_{\bullet,\bullet};Z)=\frac{3}{(-K_X)^3}\int\limits_0^{\infty}\int\limits_0^{\infty}\vol(P(u)|_Y-vZ)dv du.
\]

Assume that $0\leq u\leq 1$. Then
\begin{multline*}
\int\limits_0^{\infty}\int\limits_0^{\infty}\vol(P(u)|_Y-vZ)dv du=\int\limits_0^{2}\vol(-K_Y-vZ)dv \\
= \int\limits_0^{1}(-K_Y-vZ)^2dv+\int\limits_1^{2}(-K_Y-vZ-(v-1)(E_2+E_3))^2dv \\
= \int\limits_0^{1}(6-2v-v^2)dv+\int\limits_1^{2}(6-2v-v^2+2(v-1)^2)dv=\frac{14}{3}+\frac{4}{3}=6.
\end{multline*}

Assume that $1\leq u\leq 2$. Then
\begin{multline*}
\int\limits_0^{\infty}\int\limits_0^{\infty}\vol(P(u)|_Y-vZ)dv du=\int\limits_1^{2}\int\limits_0^{4-2u}\vol(-K_Y-vZ-(u-1)C')dvdu \\
=\int\limits_1^{2}\int\limits_0^{2-u}(-K_Y-vZ-(u-1)C')^2dvdu \\
+\int\limits_1^{2}\int\limits_{2-u}^{4-2u}(-K_Y-vZ-(u-1)C'-(u+v-2)(E_2+E_3))^2dvdu \\
=\int\limits_1^{2}\int\limits_0^{2-u}(6+6(u-1)^2-v^2-12(u-1)-2v+2(u-1)v)dvdu \\
+\int\limits_1^{2}\int\limits_{2-u}^{4-2u}(6+6(u-1)^2-v^2-12(u-1)-2v+2(u-1)v+2(u+v-2)^2)dvdu \\
=\frac{7}{6}+\frac{1}{3}=\frac{3}{2}.
\end{multline*}

Then
\begin{equation}
\label{ineq-z-(-1)-curve}
S(W^{Y}_{\bullet,\bullet};Z)=\frac{15}{16}<1.
\end{equation}

Hence, $\delta_P(X)>1$.

\subsection{$Y$ is a smooth surface, and there exist two $(-1)$-curves $E_1, E_2$ on $Y$ that pass through~$P$}
\label{sec-center-point-3}
Put $Z=E_1$. As above,
\[
\frac{3}{(-K_X)^3}\int\limits_0^{\infty}\int\limits_0^{\infty}(P(u,v)\cdot Z)^2dv du=\frac{35}{48}.
\]

Since there are exactly two $(-1)$-curves that pass through $P$, we see that
\[
\ord_P N(u,v)|_Z\leq
\begin{cases}
0, \quad \quad \quad \text{for}\ 0\leq u\leq 1, \ 0\leq v\leq 1 ,\\
v-1, \quad \quad \quad \text{for}\ 0\leq u\leq 1, \ 1\leq v\leq 2, \\
0, \quad \quad \quad \text{for}\ 1\leq u\leq 2, \ 0\leq v\leq 2-u, \\
u+v-2, \quad \quad \quad \text{for}\ 1\leq u\leq 2, \ 2-u\leq v\leq 4-2u. \\
\end{cases}
\]

Moreover, we have
\[
\ord_P N'_Y(u)|_Z\leq
\begin{cases}
0, \quad \quad \quad \text{for}\ 0\leq u\leq 1, \\
u-1, \quad \quad \text{for}\ 1\leq u\leq 2.
\end{cases}
\]

We obtain
\begin{multline*}
\int\limits_0^{\infty}\int\limits_0^{\infty}(P(u,v)\cdot Z)\cdot\ord(N'_Y(u)|_Z+N(u,v)|_Z)dv du \\
\leq \int\limits_0^1\int\limits_1^{2}(3-v)(v-1)dvdu+\int\limits_1^2\int\limits_0^{2-u}(2+v-u)(u-1)dvdu \\
+\int\limits_1^2\int\limits_{2-u}^{4-2u}(6-3u-v)(2u+v-3)dvdu=\frac{2}{3}+\frac{1}{8}+\frac{7}{24}=\frac{13}{12}.
\end{multline*}

So,
\[
\frac{6}{(-K_X)^3}\int\limits_0^{\infty}\int\limits_0^{\infty}(P(u,v)\cdot Z)\cdot\ord(N'_Y(u)|_Z+N(u,v)|_Z)dv du\leq \frac{13}{48}.
\]

Hence,
\begin{equation}
S(W^{Y,Z}_{\bullet,\bullet,\bullet};P)\leq\frac{35}{48}+\frac{13}{48}=1.
\end{equation}

Note that $\ord_Z(N(u)|_Y)=0$. Hence as above we have
\begin{equation}
\label{ineq-z-(-1)-curves}
S(W^{Y}_{\bullet,\bullet};Z)=\frac{3}{(-K_X)^3}\int\limits_0^{\infty}\int\limits_0^{\infty}\vol(P(u)|_Y-vZ)dv du = \frac{15}{16}.
\end{equation} 

We conclude that $\delta_P(X)\geq 1$. We will show that $\beta_X(E)>0$ for any prime divisor $E$ over $X$. Assume that there exists a prime divisor $E$ over $X$ such that $\beta(E)\leq 0$, and the center $C(E)$ on~$X$ contains $P$. Then $\frac{A_X(E)}{S_X(E)}\leq 1$ (and hence $=1$), so $\delta_P(X)= 1$. Consider two cases: $C(E)$ is equal to the point $P$, and $C(E)$ is a curve that contains $P$. In the former case, by Proposition \ref{Van2}, we conclude that $\delta_P(X) = \frac{1}{S_X(Y)}$. But according to Proposition \ref{St1}, we have $S_X(Y) = \frac{33}{48}$, hence we arrive at a contradiction.

In the latter case, consider the curve $C(E)$. %We claim that there exists $i$, for $1\leq i\leq 4$, and a fiber $Y_i$ of the contraction $\pi_i\colon X\to \mathbb{P}^1$, such that $Y_i$ intersects $C(E)$, $Y_i$ is smooth, and
First assume that $C(E)$ belongs to a fiber $Y=Y_i$ of $\pi_i$ for some $i$. 
%Without loss of generality, we may assume that $i=1$. 
Assume that $C(E)$ is a $(-1)$-curve. By Proposition \ref{Sing1} we see that there exists a fiber $Y'=Y_j$ of $\pi_j$ for some $j$ such that $Y'$ contains $C(E)$, and $Y'$ is smooth. Then for a general point $Q\in C(E)\subset Y'$, by the above computations we have $\delta_Q(X)>1$, and hence $\beta_X(E)>0$, which contradicts our assumption.

Assume that $C(E)$ is not a $(-1)$-curve. Then for a general point $Q\in C(E)\subset Y$, we have $\delta_Q(X)>1$, and hence $\beta_X(E)>0$, which contradicts our assumption (see \ref{sec-center-point-4} for case of singular surface~$Y$).

Now assume that $C(E)$ does not belong to a fiber of $\pi_i$ for any $i$. Then, at the general point~$Q$ of $C(E)$, all the fibers of $\pi_i$ passing through $Q$ are smooth, because each $\pi_i$ has only finitely many singular fibers. Now, assume that on all the fibers $Y_i$ passing though $Q$, the point $Q$ lies on the intersection of two $(-1)$-curves. One can note that each $(-1)$-curve that passes through $Q$ belongs to exactly three of the fibers $Y_i$. But then, on each fiber, we have $3$ $(-1)$-curves passing through~$Q$ which is not possible on a smooth sextic del Pezzo surface. This contradiction shows that $C(E)$ cannot be a curve, and hence $\beta_X(E)>0$.

\subsection{$Y$ is a singular surface}
\label{sec-center-point-4}
By Proposition \ref{prop-du-val}, we see that $Y$ has a unique singular point $P'$ of type $\mathbb{A}_1$. By Proposition \ref{Sing1} we may assume that there are no $(-1)$-curves that pass through $P$ and $P\neq P'$. Hence, in the notation of Proposition \ref{Van1}, we have $\Delta_P=0$. By Proposition \ref{Sing2} we may assume that $C' = S|_Y$ does not pass through $P$. Let $Z$ be a fiber of the contraction $\pi_1\colon X\rightarrow\mathbb{P}^1$. Then $Z^2=0$ on $Y$. Since there is no $(-1)$-curves that pass through $P$ and~$C'$ does not pass through $P$, we see that
\[
\frac{6}{(-K_X)^3}\int\limits_0^{\infty}\int\limits_0^{\infty}(P(u,v)\cdot Z)\cdot\ord_P(N'_Y(u)|_Z+N(u,v)|_Z)dv du=0.
\]

Assume that $0\leq u\leq 1$. By Proposition \ref{ZD} we have $P(u)=-K_X-uY$. So, for $0\leq v\leq 1$
\[
P(u,v)=(-K_X-uY)|_Y-vZ=-K_Y-vZ.
\]

Then
\[
\int\limits_0^{1}\int\limits_0^{1}(P(u,v)\cdot Z)^2dv du=\int\limits_0^{1}\int\limits_0^{1}4dv du=4.
\]

Assume that $1\leq v\leq 2$. Since there exists a $(-1)$-curve $E_1$ that meets $Z$. Moreover, $E_1$ passes through $P'$. We see that
\[
P(u,v)=-K_Y-vZ-2(v-1)E_1, \quad \quad N(u, v) = 2(v-1)E_1.
\]

Then
\[
\int\limits_0^{1}\int\limits_1^{2}(P(u,v)\cdot Z)^2dv du=\int\limits_0^{1}\int\limits_1^{2}(4-2v)^2dvdu=\frac{4}{3}.
\]

Note that $P(u,v)=0$ for $v>2$.

Assume that $1\leq u\leq 2$. By Proposition \ref{ZD} we have $P(u)=-K_X-uY-(u-1)S$. So, $P(u,v)=-K_Y-vZ-(u-1)C'$ for $0\leq v\leq 2-u$. Then
\[
\int\limits_1^{2}\int\limits_0^{2-u}(P(u,v)\cdot Z)^2dv du=\int\limits_1^{2}\int\limits_0^{2-u}(4-2u)^2dvdu=1.
\]

Assume that $2-u\leq v\leq 4-2u$. Then
\[
P(u,v)=-K_Y-vZ-(u-1)C'-2(u+v-2)E_1, \quad \quad N(u, v) = 2(u+v-2)E_1.
\]

Then
\[
\int\limits_1^{2}\int\limits_{2-u}^{4-2u}(P(u,v)\cdot Z)^2dv du=\int\limits_1^{2}\int\limits_{2-u}^{4-2u}(8 - 4u - 2v)^2dvdu=\frac{1}{3}.
\]

Note that $P(u,v)=0$ for $v>4-2u$. Then
\begin{equation}
\label{ineq-z-singular}
S(W^{Y, Z}_{\bullet,\bullet, \bullet};P)=\frac{3}{(-K_X)^3}\int\limits_0^{\infty}\int\limits_0^{\infty}(P(u,v)\cdot Z)^2dv du=\frac{1}{8}\left(4+\frac{4}{3}+1+\frac{1}{3}\right)=\frac{5}{6}<1.
\end{equation}

Note that $\ord_Z(N(u)|_Y)=0$. So,
\[
S(W^{Y}_{\bullet,\bullet};Z)=\frac{3}{(-K_X)^3}\int\limits_0^{\infty}\int\limits_0^{\infty}\vol(P(u)|_Y-vZ)dudv.
\]

Assume that $0\leq u\leq 1$. Then
\begin{multline*}
\int\limits_0^{\infty}\int\limits_0^{\infty}\vol(P(u)|_Y-vZ)dudv=\int\limits_0^{2}\vol(-K_Y-vZ)dv \\
= \int\limits_0^{1}(-K_Y-vZ)^2dv+\int\limits_1^{2}(-K_Y-vZ-2(v-1)E_1)^2dv \\
= \int\limits_0^{1}(6-4v)dv+\int\limits_1^{2}(6-4v+2(v-1)^2 )dv=4+\frac{2}{3}=\frac{14}{3}.
\end{multline*}

Assume that $1\leq u\leq 2$. Then
\begin{multline*}
\int\limits_0^{\infty}\int\limits_0^{\infty}\vol(P(u)|_Y-vZ)dudv
=\int\limits_1^{2}\int\limits_0^{4-2u}\vol(-K_Y-vZ-(u-1)C')dvdu \\
=\int\limits_1^{2}\int\limits_0^{2-u}(-K_Y-vZ-(u-1)C')^2dvdu \\
+\int\limits_1^{2}\int\limits_{2-u}^{4-2u}(-K_Y-vZ-(u-1)C'-2(u+v-2)E_1)^2dvdu \\
=\int\limits_1^{2}\int\limits_0^{2-u}(6+6(u-1)^2-12(u-1)-4v+4(u-1)v)dvdu \\
+\int\limits_1^{2}\int\limits_{2-u}^{4-2u}(6+6(u-1)^2-12(u-1)-4v+4(u-1)v+2(u+v-2)^2)dvdu
=1+\frac{1}{6}=\frac{7}{6}.
\end{multline*}

Then
\begin{equation}
S(W^{Y}_{\bullet,\bullet};Z)=\frac{35}{48}<1.
\end{equation}

We conclude that $\delta_P(X)>1$.

\section{Proof of main theorem}
\begin{proof}[Proof of Theorem \ref{main-theorem}]
Let $X$ be a smooth divisor of type $(1,1,1,1)$ in $(\mathbb{P}^1)^4$, so $X$ is a Fano variety of Picard rank $4$ and degree $24$. By Proposition \ref{St1}, $X$ is divisorially stable. Our aim is to show that $X$ is K-stable.

%use Propositions \ref{Van1} and \ref{Van2} to show that $\delta_P(X)>1$ for any point $P\in X$.
%In Section \ref{sec-center-point}
Let $P$ be an arbitrary point in $X$. Let $E$ be a divisor over $X$ whose center on $X$ contains the point~$P$.  We consider a surface $Y=Y_i$ containing~$P$ which is a fiber of $\pi_i$ for some $1\leq i\leq 4$. Then, we distinguish four cases.
%In the first three cases we assume that $Y$ is smooth.
First we analyse the case when there are no $(-1)$-curves on $Y$ that pass through $P$, see section \ref{sec-center-point-1}. Then we deal with the case when there exists a unique $(-1)$-curve passes through $P$, see section \ref{sec-center-point-2}. In both cases we show that $\delta_P(X)>1$, and hence $\beta_X(E)>0$.

After that, we assume that there exist two $(-1)$-curves that pass through $P$, see section \ref{sec-center-point-3}. In this case, we show that $\delta_P(X)\geq 1$. Nevertheless, we manage to show that $\beta_X(E)>0$ for any divisor $E$ over $X$ such that the center of $E$ contains $P$.

Finally, if all four fibers of $\pi_i$, for $1\leq i\leq 4$, passing through $P$ are singular, by Proposition~\ref{prop-du-val} we know that each of them is a sextic del Pezzo surface with exactly one du Val singularity of type~$\mathbb{A}_1$. We consider this case in section~\ref{sec-center-point-4}, and show that $\delta_P(X)>1$, and hence $\beta_X(E)>0$.

This implies that the variety $X$ is K-stable.
%$S(W^{Y}_{\bullet,\bullet};Z)\leq 1$ and $S(W^{Y,Z}_{\bullet,\bullet,\bullet};P)\leq 1$. Using the fact that $S_X(Y)<1$ for any surface $Y$ in $X$ proven in Proposition \ref{St1}, by Proposition \ref{Van1} we conclude $\delta_P(X)>1$.
%Hence $\delta(X)>1$, and the variety $X$ is K-stable.
%In section \ref{sec-center-curve} we assume that $Z$ is a curve. Then there exists a projection, say $\pi_{1, 2}$, such that $\pi_{1,2}(Z)$ is a curve on $\mathbb{P}^1\times\mathbb{P}^1$. We call this curve $Z'$, and we may assume that it has bidegree $(a,b)$ for $a\geq b\geq 0$. Then we consider a birationally ruled surface $Y=\pi_{1,2}^{-1}(Z')$. In section \ref{sec-center-curve-1}, we consider the case when $b=0$. Then in section \ref{sec-center-curve-2} we treat the general case $b>0$.
The theorem is proven.
\end{proof}

\end{document}